\documentclass[11pt]{amsart}
\usepackage{geometry}
\usepackage{graphicx}
\usepackage{amssymb}
\usepackage{amsmath}
\usepackage{color}
\newtheorem{theorem}{Theorem}
\newtheorem{definition}[theorem]{Definition}
\newtheorem{lemma}[theorem]{Lemma}
\newtheorem{corollary}[theorem]{Corollary}
\newtheorem{proposition}[theorem]{Proposition}
\newtheorem{remark}[theorem]{Remark}

\newcommand{\R}{\mathbb R}
\newcommand{\M}{\mathbb M}
\numberwithin{theorem}{section}
\numberwithin{equation}{section}
\begin{document}
\title[Penrose with respect to static spacetimes]{A quasi-local Penrose inequality for the quasi-local energy with static references}
\author{Po-Ning Chen}
\date{}
\begin{abstract} 
The positive mass theorem is one of the fundamental results in general relativity. It states that, assuming the dominant energy condition, the total mass of an asymptotically flat spacetime is non-negative.
The Penrose inequality provides a lower bound on mass by the area of the black hole and is closely related to the cosmic censorship conjecture in general relativity. In \cite{Lu-Miao}, Lu and Miao proved a quasi-local Penrose inequality for the quasi-local energy with reference in the Schwarzschild manifold. In this article, we prove a quasi-local Penrose inequality for the quasi-local energy with reference in any spherically symmetric static spacetime. 
\end{abstract}
\thanks{The author would like to thank Siyuan Lu, Pengzi Miao, Mu-Tao Wang, Ye-Kei Wang and Xiangwen Zhang for helpful discussions. The author is supported by NSF grant DMS-1308164 and Simons Foundation collaboration grant \#584785. Part of this work was carried out when the author was visiting the National Center of Theoretical Sciences at National Taiwan University in Taipei, Taiwan.} 
\maketitle
\section{Introduction}
In general relativity, a spacetime is a 4-dimensional manifold $N$ with a Lorentz metric $g_{\alpha \beta}$ and a symmetric 2-tensor $T_{\alpha \beta}$. The metric represents the gravitation field while the symmetric 2-tensor represents the stress-energy density of matter in the spacetime. The Einstein equation relates the geometric information to the matter field. It reads
\[ R_{\alpha \beta} - \frac{R}{2} g_{\alpha \beta} = 8 \pi T_{\alpha \beta},  \]
where $R_{\alpha \beta}$ and $R$ are the Ricci curvature and the scalar curvature of the metric, respectively. Throughout this article, we assume that the matter field satisfies the dominant energy condition. Namely, for any time like field $e_0$, $T(e_0,e_0) \ge 0$ and $T(e_0,\cdot)$  is a non-space-like co-vector.

An initial data set consists of $(M, g_{ij}, k_{ij})$ where $g_{ij}$ is the induced metric and $k_{ij}$ is the second fundamental form of the hypersurface $M$ in the spacetime N. For an asymptotically flat initial data set, Arnowitt, Deser, and Misner defined the total Arnowitt-Deser-Misner (ADM) energy-momentum $(E,P_i)$ using the asymptotic symmetry at infinity \cite{ADM}. One of the fundamental results in general relativity is the positive mass theorem by Schoen-Yau \cite{SY1,SY2} and Witten \cite{W}. It states that the ADM energy-momentum satisfies 
\[  E \ge \sqrt{\sum_i P_i^2}, \]
if the spacetime satisfies the dominant energy condition. Moreover, the equality holds if and only if the spacetime is the Minkowski space.

An important situation is the time-symmetric initial data sets where $k_{ij}= 0$. In this case, the dominant energy condition implies that the scalar curvature of $M$ is non-negative. Moreover, the ADM linear momentum vanishes and the ADM energy is the same as the ADM mass.

The Penrose inequality is a conjuncture on the lower bound of the mass of an isolated system in terms of the total area of its black holes \cite{Penrose}. Its validity is related to the cosmic censorship conjecture on the global structure of solutions to the Einstein equation. An important special case is the Riemannian Penrose inequality for time-symmetric asymptotically flat initial data sets. Namely, if $(M,g)$ is an asymptotically flat Riemannian 3-manifold with nonnegative scalar curvature and ADM mass $m$, and $A$ is the area of the outermost minimal surface (possibly with multiple connected components), then the Riemannian Penrose inequality asserts that 
\[
m \ge \sqrt{\frac{A}{16 \pi}}.
\]

Huisken and Ilmanen proved the inequality where $A$ is the area of the largest component of the outermost minimal surface \cite{Huisken-Ilmanen}. Their proof used the monotonicity of the Hawking mass under the inverse mean curvature flow. Bray gave the complete proof of the above inequality using a conformal flow of metrics \cite{Bray}. 

Following the proof of the positive mass theorem, there have been many different attempts to define quasi-local mass. See \cite{Szabados} for earlier works on defining quasi-local mass and other quasi-local conserved quantities. Here we briefly review several notions of quasi-local mass using the Hamilton-Jacobi approach and their positive mass theorems. The basic idea is that we should compare the surface Hamiltonian \cite{BY,HH} of the surface inside the spacetime $N$  to the surface Hamiltonian of suitable configurations of the same surface inside a chosen reference spacetime. 

For time-symmetric initial data sets, Brown and York defined the Brown-York mass in \cite{BY}. Suppose $(M,g)$ is a 3-manifold such that $\Sigma$, the boundary of $M$, has a positive Gauss curvature. From the solution of Weyl’s isometric embedding problem by Nirenberg \cite{Nirenberg}
 and independently, Pogorelov \cite{Pogorelov}, there exists an unique isometric embedding of $\Sigma$ into $\R^3$. Let $H$ and $H_0$ be the mean curvatures of $\Sigma$ in $M$ and $\R^3$, respectively. 

 The Brown-York mass of $\Sigma$ is 
\[
m_{BY}(\Sigma) = \frac{1}{8 \pi} \int_{\Sigma} H_0 - H .
\]
 Shi and Tam proved the following positive mass theorem for the Brown-York mass:
\begin{theorem}
Suppose the scalar curvature of $M$ is non-negative and $H >0$. Then the Brown-York mass of $\Sigma$ is non-negative, and it equals zero if only if $M$ is flat.
\end{theorem}
The proof by Shi and Tam used Bartnik's quasi-spherical metric \cite{Bartnik} to construct an extension of $(M,g)$ to an asymptotically flat manifold on which the Brown-York mass is non-increasing. The positivity of the Brown-York mass follows from the positive mass theorem (for manifolds with a certain type of singularities).

Liu and Yau defined the Liu-Yau mass which does not depend on the choice of the hypersurface $M$ \cite{Liu-Yau1,Liu-Yau2}. Suppose $\Sigma$ is an embedded 2-sphere with positive Gauss curvature that bounds a spacelike hypersurface in a spacetime $N$. The Liu-Yau mass is
\[
m_{LY}(\Sigma) = \frac{1}{8 \pi} \int_{\Sigma} H_0 - |H|, 
\]
where $H$ is the mean curvature vector of $\Sigma$ in the spacetime $N$.

Liu and Yau proved the following positive mass theorem:
\begin{theorem}
Suppose $N$ satisfies the dominant energy condition and $|H| >0$. Then the Liu-Yau mass of $\Sigma$ is non-negative, and it equals zero if only if $N$ is the Minkowski space along $\Sigma$.
\end{theorem}
In \cite{Wang-Yau1,Wang-Yau2}, Wang and Yau defined the Wang-Yau quasi-local mass which captures the full information of the energy-momentum of a surface in a Lorentzian manifold. An energy is assigned to each isometric embedding of the surface into the Minkowski space and a timelike vector in Minkowski space. Such a pair is considered to be an observer. Wang and Yau showed that the assigned energy is always non-negative and the Wang-Yau quasi-local mass is the minimum among all possible assigned energy.

In \cite{Chen-Wang-Wang-Yau}, a new quasi-local energy is defined using any static spacetime as the reference. This work generalizes the Wang-Yau quasi-local mass which used the Minkowski space as the reference. The new quasi-local energy provides a measurement of how far a spacetime deviates away from the exact static solution. 

The static slice of a static spacetime is referred to as a static manifold. It consists of a triple $(\M, \bar g, V)$ where $\bar g$ is the induced metric on the 3-manifold $\M$ and $V$ is a function (the static potential) on $\M$ satisfying 
\[
\bar \Delta V  \bar g  - \bar \nabla^2 V +  V \bar Ric \ge 0,
\]
where $\bar \nabla$, $\bar \Delta$ and $\bar Ric$ are the Levi-Civita connection, Laplacian and Ricci curvature of the metric $\bar g$, respectively. 

For an isometric embedding into $\M$, the quasi-local energy is simply
\[
\frac{1}{8 \pi} \int_{\Sigma} V (H_0 - |H|),
\]
where $H_0$ is the mean curvature of the image in $\M$. For this article, we focus on the time-symmetric case. Namely, $\Sigma$ bounds a hypersurface $M$ with vanishing second fundamental form. Under this assumption, the quasi-local energy is then
\[
\frac{1}{8 \pi} \int_{\Sigma} V (H_0 - H),
\]
where $H$ is the mean curvature of $\Sigma$ in $M$.

A natural candidate for the reference space is the Schwarzschild spacetime. In \cite{Li-Wang}, Li and Wang found the following lower bound 
\[
\frac{1}{8 \pi} \int_{\Sigma} V (H_0 - H) \ge -m
\]
assuming the image of the isometric embedding is sufficiently close to a coordinate sphere in the Schwarzschild manifold. It is natural to ask whether we can improve the lower bound by the area of the enclosed horizon similar to the Penrose inequality. In \cite{Lu-Miao}, Lu and Miao proved the following quasi-local Penrose inequality (here we state their result for 3-manifold only):
\begin{theorem}  
Let  $(M, g)$ be a compact, connected,  orientable, 3-manifold  with nonnegative scalar curvature, with boundary $\partial M$. 
Suppose the boundary is the union of  $ \Sigma $ and $ \Sigma_h$, where 
\begin{enumerate}
\item[(i)]   $\Sigma$ has positive mean curvature $H > 0$; and  
\item[(ii)]  $\Sigma_h$,  if nonempty, is a minimal surface and there are no other closed  minimal surfaces in $ (M, g)$.  
\end{enumerate}
Suppose $ \Sigma$ is isometric to a convex surface in the Schwarzschild manifold such that $\bar Ric (\nu,\nu) \le 0$.
Then 
\[ 
 \frac{1}{8 \pi } \int_\Sigma V (H_0 - H) \, d \sigma  \ge \sqrt{ \left(\frac{|\Sigma_h|}{16 \pi}\right) } - m.
\]
where $m$ is the mass of the reference Schwarzschild manifold.
\end{theorem}
In \cite{Lu-Miao}, it is also proved that the equality of the quasi-local Penrose inequality implies that the two mean curvatures are equal. See also \cite{Chen-Zhang,Lu-Miao1,Shi-Wang-Yu} on the equality case of this quasi-local Penrose inequality. 

In this article, our goal is to prove a quasi-local Penrose inequality for the quasi-local energy with an isometric embedding into the static slice of a spherically symmetric static spacetime satisfying the dominant energy condition (we will refer to such a manifold as a spherically symmetric static manifold). One such example is the Reissner-Nordstrom manifold, which is the 3-manifold with metric
\[
\bar g = \frac{d r^2}{1 - \frac{2m}{r} + \frac{e^2}{r^2}} + r^2 dS^2
\]
and the static potential 
\[
V = \sqrt{1 - \frac{2m}{r} + \frac{e^2}{r^2}}.
\]
It is the static slice in the Reissner-Nordstrom spacetime with metric 
\[
-(1 - \frac{2m}{r} + \frac{e^2}{r^2}) dt^2 + \frac{d r^2}{1 - \frac{2m}{r} + \frac{e^2}{r^2}} + r^2 dS^2.
\]

Here, we state the main theorem of this article when the surface is isometrically embedded into the Reissner-Nordstrom manifold
\begin{theorem}  \label{thm_main}
Let  $(M, g)$ be a compact, connected,  orientable, 3-manifold  with nonnegative scalar curvature, with boundary $\partial M$. 
Suppose the boundary is the union of  $ \Sigma $ and $ \Sigma_h$, where 
\begin{enumerate}
\item[(i)]   $\Sigma$ has positive mean curvature $H > 0$; and  
\item[(ii)]  $\Sigma_h$,  if nonempty, is a minimal surface and there are no other closed  minimal surfaces in $ (M, g)$.  
\end{enumerate}
There exists constants $C_6$ and $C_7$ depending only on the mass $m$ and charge $e$ of the reference Reissner-Nordstrom manifold
such that if $ \Sigma$ is isometric to a convex surface in the Reissner-Nordstrom manifold with principal curvature $\kappa_a$  and
\[
\begin{split}
 -  \bar Ric (\nu,\nu) > & 0\\
\min \kappa_a  > &  \frac{C_6}{r^2}\\
r  > &  C_7,
\end{split}
\]
then 
\[ 
 \frac{1}{8 \pi } \int_\Sigma V (H_0 - H) \, d \sigma  \ge \sqrt{ \left(\frac{|\Sigma_h|}{16 \pi}\right) } - m.
\]
\end{theorem}
In Section 5, we will also prove a similar result for isometric embeddings into any asymptotically flat and spherically symmetric static manifold.

One of the key steps in the positive mass theorem for the Brown-York mass by Shi and Tam and the quasi-local Penrose inequality by  Lu and Miao is to solve the prescribed scalar curvature equation 
\begin{equation}\label{vacuum_constraint}
R = 0
\end{equation}
on the exterior region of the isometric embedding using Bartnik's quasi-spherical metric.  In both cases, an appropriate quasi-local mass is monotonic on the solution of \eqref{vacuum_constraint}. It is crucial for us to find an analogous equation for the static manifold.

The structure of the paper is as follows: in Section 2, we review the proof of the positivity of Brown-York mass by Shi and Tam and the quasi-local Penrose inequality by Lu and Miao and give an outline of our proof.  In Section 3, we introduce the prescribed scalar curvature equation  \eqref{prescribed} using Bartnik's quasi-spherical metric \eqref{perturbed metric} and derive a monotonicity formula (see Corollary \ref{cor-1}) for the quasi-local mass on the solution of the prescribed scalar curvature equation. In Section 4, we study the prescribed scalar curvature equation in more details and prove a sufficient condition about the foliation (see \eqref{condition_foliation}) so that the solution to \eqref{prescribed}  is a smooth and asymptotically flat manifold with non-negative scalar curvature. In particular, this reduces the quasi-local Penrose inequality to show that the unit normal flow creates a foliation satisfying the condition given in \eqref{condition_foliation}. Up to Section 4, the result holds for any static manifold. In Section 5, we find a sufficient condition on the initial surface $\Sigma_0 $ so that the unit normal flow resulted in a foliation satisfying \eqref{condition_foliation} in a spherically symmetric static manifold.

\section{Review of the proofs of Shi-Tam and Lu-Miao}
Let us first recall the proof of the positivity for the Brown-York mass by Shi and Tam in \cite{Shi-Tam} and the proof of the quasi-local Penrose inequality by Lu and Miao \cite{Lu-Miao}. This helps us to illustrate the main strategy of our proof since we adopt several ideas and methods from the two articles. 

Suppose $(M,g)$ is a 3-manifold with a non-negative scalar curvature such that $\Sigma$, the boundary of $M$, is a topologically $S^2$ with positive Gauss curvature and positive mean curvature. 
Let $\Sigma_0$ be the image of the isometric embedding of $\Sigma$ into $\R^3$.
Since the $ \Sigma_0 $ is  convex in $ \R^3$, we write the Euclidean metric on  the exterior of $ \Sigma $, $M_{ext}$,  as  
\[ \delta = d s^2 + \sigma_s, \] 
where $\sigma_s $ is  the induced metric  on the  hypersurface $\Sigma_s$  with a fixed Euclidean distance $s$  to $ \Sigma$. Consider a new metric on $M_{ext}$ of the form
\[ g_u = u^2 d s^2 +\sigma_s \]
where $u$ will be determined by the prescribed scalar curvature equation
\[
R(g_u)= 0.
\]
This is a parabolic equation on $u$ and by the maximum principle, for any initial value of $u$ on $ \Sigma_0$, $g_u$ will be asymptotically flat. In particular, we may choose $u$ such that on $\Sigma_0$, 
\[   u =  \frac{H}{H_0}.\]

Let $H_0$ and $H_u$ denote the mean curvature of $\Sigma_s$ with respect to $\delta$ and $g_u$. The prescribed scalar curvature equation implies that 
\[ 
\frac{1}{ 8 \pi }  \int_{\Sigma_s} (H_0 - H_u ) \, d \sigma_s 
 \]
is nonincreasing  in $s$.  Moreover,
\[
\lim_{s \to \infty} \frac{1}{ 8 \pi }  \int_{\Sigma_s} (H_0 - H_u ) \, d \sigma_s  = m_{ADM} (g_u).
\]
It follows that 
\[
\frac{1}{8 \pi} \int_{\Sigma} H_0 - H  \ge  m_{ADM} (g_u).
\]

By gluing $(M,g) $ and $(M_{ext}, g_u)$ along their common boundary $\Sigma$ 
and applying  the positive mass theorem (which is valid under the assumption that $H_u=H$ along $\Sigma$), it follows that 
\[
\frac{1}{8 \pi }  \int_\Sigma (H_0 - H) \, d \sigma   \ge 0.
\]

To summarize, the following two key features of the prescribed scalar curvature equation 
\[
R(g_u) = 0
\]
via Bartnik's quasi-spherical metric played fundamental roles in the proof of the positive mass theorem:
\begin{enumerate}
\item The prescribed scalar curvature equation gives rises to an asymptotically flat manifold with non-negative scalar curvature.
\item The monotonicity formula for the quasi-local mass on the solution of the prescribed scalar curvature equation.
\end{enumerate}

In \cite{Lu-Miao}, Lu and Miao proved a quasi-local Penrose inequality using isometric embedding into the Schwarzschild manifold as the reference. New difficulties arise for their proof. For example, the static potential affects both the monotonicity formula and the prescribed scalar curvature equation. Also, they have to use the Riemannian Penrose inequality for gluing satisfying the condition $H_u=H$. 

To overcome these difficulties, they used an inverse curvature flow to construct a convex foliation with
\[
\bar Ric (\nu,\nu) \le 0.
\]
Assuming the above condition, they show that the solution to the prescribed scalar curvature equation 
\[
R(g_u) = 0
\]
is a smooth asymptotically flat manifold on which the quasi-local mass is non-increasing.

By the gluing construction and the monotonicity formula, the quasi-local Penrose inequality now follows from the Riemannian Penrose inequality for an asymptotically flat manifold with corners. This particular version of the Riemannian Penrose inequality is proved in \cite{Miao} using a density theorem. 

In our case, we consider the quasi-local mass associated to an isometric embedding into a static manifold. Given any foliation of the static manifold, we consider the following function $T$ defined by  
\[
(\bar{\Delta}V)-\bar{D}^2 V(\nu,\nu)+V\bar{Ric}(\nu,\nu)=TV.
\]

Our key observation is that if the unit normal flow of $\Sigma$ in the static manifold gives rise to a convex foliation satisfying 
\begin{equation} \label{condition_foliation}
\frac{\partial V}{\partial \nu}, \ det (A_0)+ \frac{T}{2} -  \bar Ric(\nu,\nu) > 0   \  \mathrm{and} \  det (A_0) >  \frac{T}{2},
\end{equation}
where $A_0$ is the second fundamental form of the surfaces, then the following prescribed scalar curvature equation
\[
\tilde R =  \bar R + (\frac{1}{u^2}-1)T
\] 
will satisfy both the key properties above as in the proof of Shi and Tam for the Brown-York mass. Then we apply the Penrose inequality for asymptotically flat manifolds with corners established in \cite{Miao} to prove a quasi-local Penrose inequality for the quasi-local energy with reference in the static manifold. 

To find a condition on $\Sigma_0$ such that the unit normal flow would satisfy \eqref{condition_foliation}, we utilize that any spherically symmetric manifold is conformally flat. The conformal flatness allows us to study a corresponding flow of surfaces in $\R^3$ and use the evolution formulae for the principal curvature and the support function in  $\R^3$ as well as the formula for the principal curvature under a conformal change.

\section{Monotonicity formula in a static background}  \label{sec-mono-static}
In this section, we consider the variation of the quasi-local energy with reference in a static manifold. Namely, we have a triple $(\M, \bar g,V)$ such that 
\[
\bar \Delta V  \bar g  - \bar \nabla^2 V +  V \bar Ric \ge 0.
\]
where $\bar \nabla$, $\bar \Delta$ and $\bar Ric$ are the Levi-Civita connection, Laplacian and Ricci curvature of the metric $\bar g$, respectively.

Consider a family of embedded hypersurfaces $\{\Sigma_s \}$  evolving in $(\M, \bar g)$ according to 
\begin{align}\label{formula for flow}
\frac{\partial X}{\partial s}=f\nu , 
\end{align}
where $X$ denotes points in $\Sigma_s$,  $f >0$ denotes  the speed  of the flow, and $\nu$ is a  unit normal to $\Sigma_s$. Denote the region swept by $\{ \Sigma_s \}$ as $U$ and we consider the following function $T$ on $U$ defined by 
\[
(\bar{\Delta}V)-\bar{D}^2 V(\nu,\nu)+V\bar{Ric}(\nu,\nu)=TV.
\]

The following lemma is a simple consequence of the dominant energy condition. 
\begin{lemma} \label{positive_scalar}
The dominant energy condition implies that the scalar curvature $\bar R$ of $\bar g$ and $T$ satisfies
\[
\bar R \ge T \ge 0.
\]
\end{lemma}
\begin{proof}
Consider the Einstein tensor
\[ G_{\alpha \beta} =  R_{\alpha \beta} - \frac{R}{2} g_{\alpha \beta} .  \]
Recall that (see for example \cite{Beig})
\[
\begin{split}
\bar R = & 2 G(e_0,e_0) \\
T =  &   G(e_0,e_0)  - G(\nu,\nu)
\end{split}
\]
Thus, from the dominant energy condition, we conclude that 
\[
\begin{split}
T =  &   G(e_0+ \nu ,e_0 - \nu ) \ge 0  \\
\bar R  - T = & \frac{1}{2} \left (  G(e_0 + \nu,e_0+ \nu)  +  G(e_0 - \nu,e_0-  \nu)  \right ) \ge 0.
\end{split}
\]
\end{proof}

The metric  $\bar g $ over the region $U$ can be written as
\begin{align}\label{metric by foliation}
\bar{g}=f^2ds^2+ \sigma_s ,
\end{align}
where $\sigma_s$ is the induced metric of $\Sigma_s$.

To establish a monotonicity formula for the quasi-local energy with reference in $(\M, \bar g)$, we  consider  another metric 
\begin{align}\label{perturbed metric}
\tilde g=u^2 f^2dt^2+\sigma_s , 
\end{align}
where $u > 0 $ is a function on $U$. 
We impose the following condition on the scalar curvature $\tilde R$  of $\tilde g $:
\begin{equation}\label{prescribed}
\tilde R = \bar R  + (\frac{1}{u^2}-1) T. 
\end{equation}
\begin{proposition}  \label{prop-mono}
Under the above notations and assumptions, 
\begin{align*}
  \frac{d}{d s}\left( \int_{\Sigma_t} V ( H_0 - \tilde H ) d \sigma_s \right)= -  \int_{\Sigma_s} 
\frac{1}{u}( u - 1 )^2 f ^2  H_0  \frac{\partial V}{\partial \nu} d \sigma_s  - \int_{\Sigma_s}   
    \frac{V}{u}( u - 1 )^2 f ^2 (det(A_0)- \frac{T}{2})   d \sigma_s , 
\end{align*}
where $H_0$ and $\tilde H$ are the mean curvature of $\Sigma_s$ with respect to $\bar{g}$ 
and $\tilde g$, respectively, and  $A_0$  is the second fundamental form of $\Sigma_s$ with respect to  
$\bar{g}$.
\end{proposition}
\begin{proof}
Let $\tilde A$ be the second fundamental form of $\Sigma_s$ with respect to $\tilde g$. Recall that 
\begin{align}\label{formula H,A}
\tilde H=\frac{H_0}{u}, \quad \tilde A=\frac{A_0}{u}.
\end{align}
By the second variation formula, 
\begin{align}\label{H}
\frac{\partial}{\partial s} H_0 =  - \Delta f - f( | A_0|^2 + \bar{Ric}(\nu, \nu)) 
\end{align}
\begin{align}\label{H_u}
\frac{\partial}{\partial s}  \tilde H=  - \Delta (fu) - fu ( | \tilde A |^2 + Ric_{\tilde g} (\tilde \nu, \tilde \nu)) ,
\end{align}
where $\Delta$ is the Laplacian on $(\Sigma_s, \sigma_s)$, $Ric_{\tilde g}$ is the Ricci curvature of  $ \tilde g$ and $\tilde \nu$ is the unit normal of $\Sigma_s$ with respect to $\tilde g$ .

Let $det(A_0)$ and $ det(\tilde A) $ be the second elementary symmetric functions of the principal curvatures of $\Sigma_s$ in $(\M, \bar g)$ and $(\M, \tilde g)$, respectively. 
By the Gauss equation,
\begin{align}\label{Gauss-1}
det(A_0)=K -\frac{\bar{R}}{2}+\bar{Ric}(\nu,\nu),\quad det(\tilde A)=K- \frac{\tilde R}{2}+Ric_{\tilde g}(\tilde \nu,\tilde \nu) .
\end{align}
We have
\begin{equation}\label{Ric-1}
\begin{split}
Ric_{\tilde g}(\tilde \nu,\tilde \nu) = & \bar{Ric}(\nu,\nu)+det(\tilde A)-det(A_0)+  \frac{\tilde R - \bar R }{2}\\
= & \bar{Ric}(\nu,\nu)+det(A_0)(\frac{1}{u^2}- 1) +   (\frac{1}{u^2}-1) \frac{T}{2}. 
\end{split}
\end{equation}
Putting  (\ref{H}), (\ref{H_u}) and (\ref{Ric-1}) together, we have
\begin{align*}
   & \frac{\partial}{\partial s} ( H_0  - \tilde H) \\
= & \Delta ( u f- f ) - f( | A_0 |^2 + \bar{Ric} (\nu, \nu)) + f u  ( | \tilde A |^2 + Ric_{\tilde g} (\nu, \nu)) \\
=& \Delta ( f(u-1) )+\bar{Ric} (\nu, \nu) f(u-1) + | A_0 |^2 f (\frac{1}{u}-1)+ f (\frac{1}{u} - u) (det(A_0) + \frac{T}{2}).
\end{align*}
Using the formula $\frac{\partial}{\partial s}d \sigma=fH_0d \sigma$ and integrating  by part, we have
\begin{align*}
 & \frac{d}{ds}\left(\int_{\Sigma_t} V( H_0  - \tilde H ) d \sigma_s\right) \\
=&\int_{\Sigma_s}f\frac{\partial V}{\partial \nu} H_0 
(1- \frac{1}{u})d \sigma_s+\int_{\Sigma_s} V H_0 ^2 (1- \frac{1}{u}) f d \sigma_s + \int_{\Sigma_s}  V f  (\frac{1}{u}-u) \frac{T}{2}  d \sigma_s \\
&+\int_{\Sigma_s}f(u-1) \left(\Delta V  +V \bar{Ric}(\nu, \nu)\right) d \sigma_s +\int_{\Sigma_s}\left( V|A_0|^2 f (\frac{1}{u}-1)+V det(A_0) f (\frac{1}{u}- u)\right) d \sigma_s \\
=&\int_{\Sigma_s}f(u-1)\left(\Delta V + V \bar{Ric} (\nu, \nu)+ \frac{1}{u}H_0 \frac{\partial V}{\partial \nu}\right) d \sigma_s +  \int_{\Sigma_s}  V f  (\frac{1}{u}-u) \frac{T}{2}  d \sigma_s \\
&-\int_{\Sigma_s}V det(A_0) \frac{(u-1)^2}{u}d \sigma_s .
\end{align*}
Moreover, 
\begin{align*}
\Delta V+V\bar{Ric}(\nu, \nu)= \bar{\Delta}V-\bar{D}^2V(\nu,\nu)- H_0 \frac{\partial V}{\partial \nu}+ V \bar{Ric}(\nu, \nu)
= -H_0 \frac{\partial V}{\partial \nu}  + TV.
\end{align*}
Therefore, we conclude 
\begin{align*}
  &  \frac{d}{ds}\left(\int_{\Sigma_s} V( H_0  - \tilde H ) d \sigma_s \right) \\
=&\int_{\Sigma_s}f (u-1)(-1+ \frac{1}{u}) H_0 \frac{\partial V}{\partial \nu} d \sigma_s - \int_{\Sigma_s}V det(A_0) \frac{(u-1)^2}{u} fd \sigma_s  \\
   &+  \int_{\Sigma_s}  V Tf  \left ( \frac{1}{2}(\frac{1}{u}-u)  +(u-1) \right) d \sigma_s \\
=&-\int_{\Sigma_s}\frac{(u-1)^2}{u} f H_0 \frac{\partial V}{\partial \nu} d \sigma_s -\int_{\Sigma_s}V (det(A_0) - \frac{T}{2}) \frac{(u-1)^2}{u} f d \sigma_s .
\end{align*}

\end{proof}

\begin{corollary}\label{cor-1}
Suppose $(\M, \bar g)$ is foliated by $ \{ \Sigma_s \}$ such that 
\begin{equation} \label{eq-monotone}
 \frac{\partial V}{\partial \nu} >  0, \  H_0  > 0 \ \mathrm{and} \ det(A_0) >  \frac{T}{2} .
\end{equation}
Then  $ \displaystyle \int_{\Sigma_s} V ( H_0 - \tilde H) d \sigma_s $ is nonincreasing in $s$
and it is a constant if and only if $ u=1 $.
\end{corollary}
\begin{remark}
For a vacuum initial data set, $T \equiv 0$. In this special case, the results here are precisely the ones in \cite[Section 2]{Lu-Miao}. 
\end{remark}
\section{Foliation and the quasi-local Penrose inequality}
In this section, we study the prescribed scalar curvature equation \eqref{prescribed}. We now assume that $(\M, \bar g)$ is asymptotically flat and $\Sigma_s$
foliated $\M$ outside a compact set. See \cite[Definition 1.1]{Fan-Shi-Tam} for the precise assumption on the  asymptotically flatness. We will assume that the static potential approaches to $1$ at infinity with a comparable rate as well. Namely, for some $\tau > \frac{1}{2}$
\[
V= 1 + O(r^{-\tau}).
\]

We now assume that $f=1$. Namely, we assume that the metric $\bar g$ is 
\[
\bar g = ds^2 + \sigma_s
\]
\begin{proposition}  \label{prop-af}
Suppose $(\M, \bar g)$ is asymptotically flat and the unit normal foliation $\Sigma_s$ satisfies
\[
 det(A_0)  + \frac{T}{2} - \bar Ric(\nu,\nu) >  0
\]
then the solution to the  prescribed scalar curvature equation 
\[
\tilde R = \bar R  + (\frac{1}{u^2}-1) T. 
\]
for the metric
\[
\tilde g = u^2 ds^2  + \sigma_s
\] is a smooth and asymptotically flat manifold $\tilde g$ for any initial value $u> 0$ on $\Sigma_0$.
\end{proposition}
\begin{proof}
Recall that the $\tilde R$ is given by
\[
H_0 \frac{\partial u}{\partial s} = u^2 \Delta u + (u - u^3) K - \frac{1}{2} u \bar R +  \frac{1}{2} u^3 \tilde R
\]
As a result, the  prescribed scalar curvature equation  is the same as
\[
H_0 \frac{\partial u}{\partial s} =  u^2 \Delta u + (u - u^3) K - \frac{1}{2} u \bar R +  \frac{1}{2} u^3 \bar R  + \frac{1}{2}(u-u^3) T
\]
Using the Gauss equation, we conclude that 
\[
H_0 \frac{\partial u}{\partial s} = u^2 \Delta u  + (u-u^3) (det(A_0)  - \bar Ric  (\nu,\nu)  + \frac{T}{2})
\]
The assumption implies that in the last term, $(u-u^3)$ is multiplied to a positive term. It follows that the solution $u$ exists for all $s$ and 
\[
\lim_{s \to \infty} u =1
\]
by the maximum principle. This finishes the proof of the proposition. See for example Section 2 of \cite{Shi-Tam} for more details.
\end{proof}
\begin{remark}
In particular, 
\[
u = 1 + O(s^{-1})
\]
which we will need later.
\end{remark}

As a result, we conclude the following
\begin{proposition}  \label{local_penrose_pre}
Let  $(M, g)$ be a compact, connected,  orientable, 3-manifold  with nonnegative scalar curvature, with boundary $\partial M$. 
Suppose the boundary is the union of  $ \Sigma $ and $ \Sigma_h$, where 
\begin{enumerate}
\item[(i)]   $\Sigma$ has positive mean curvature $H > 0$; and  
\item[(ii)]  $\Sigma_h$,  if nonempty, is a minimal surface and there are no other closed  minimal surfaces in $ (M, g)$.  
\end{enumerate}
Suppose $ \Sigma$ is isometric to a convex surface in an asymptotically flat static manifold  $\M$ such that the unit normal flow gives a convex foliation satisfying 
\[
\frac{\partial V}{\partial \nu}>0, \ det(A) + \frac{T}{2} - \bar Ric (\nu,\nu) \ge 0   \ \mathrm{and} \  det(A) >  \frac{T}{2}.
\]
Then 
\[ 
 \frac{1}{8 \pi } \int_\Sigma V (H_0 - H) \, d \sigma  \ge \sqrt{ \left(\frac{|\Sigma_h|}{16 \pi}\right) } - m 
\]
where $m$ denote the ADM mass of $\M$.
\end{proposition}
\begin{proof}
Given the isometric embedding into $\M$, we set
\[
u = \frac{H}{H_0}
\]
on $\Sigma_0$. Using the unit normal flow, we write the metric $\bar g$ as
\[
\bar g = ds^2 + \sigma_s
\]
and solve the prescribed scalar curvature equation 
\[
\tilde R = \bar R  + (\frac{1}{u^2}-1) T. 
\]
for the metric
\[
\tilde g = u^2 ds^2  + \sigma_s.
\]
By Proposition \ref{prop-af}, $\tilde g$ is smooth and asymptotically flat. $\tilde R \ge 0$ by Lemma \ref{positive_scalar}. Using the initial value on $\Sigma_0$, we have
\[
 \frac{1}{8 \pi } \int_\Sigma V (H_0 - H) \, d \sigma  =  \frac{1}{8 \pi } \int_{\Sigma_0} V (H_0 - \tilde H) \, d \sigma 
\]
By Corollary \ref{eq-monotone}, we get
\[
 \frac{1}{8 \pi } \int_{\Sigma_0} V (H_0 - \tilde H) \, d \sigma  \ge  \lim_{s\to \infty}\frac{1}{8 \pi}  \int_{\Sigma_s} V (H_0 - \tilde H) \, d \sigma_s
\]
Note that 
\[
\lim_{s\to \infty}\frac{1}{8 \pi}  \int_{\Sigma_s} V (H_0 -\tilde H) \, d \sigma_s  =   \lim_{s\to \infty}\frac{1}{8 \pi}  \int_{\Sigma_s}  (H_0 - \tilde H) \, d \sigma_s
\]
since $V= 1+ O(s^{-\tau})$, $H_0 = \frac{2}{r} + O(s^{-1-\tau})$ for some $\tau > \frac{1}{2}$, and $u=1 + O(s^{-1}) $. For $s$ sufficiently large, the Gauss curvature for $\sigma_s$ is positive and thus there exists a unique 
isometric embedding of $\Sigma_s$  into $\R^3$. Let $\bar H$ be the mean curvature of the isometric embedding of $\Sigma_s$ into $\R^3$, we compute
\[
\begin{split}
     & \lim_{s\to \infty}\frac{1}{8 \pi}  \int_{\Sigma_s}  (H_0 - \tilde H) \, d \sigma_s  \\
  = & \lim_{s\to \infty}\frac{1}{8 \pi}  \int_{\Sigma_s}  (\bar H - \tilde H) \, d \sigma_s + \lim_{s\to \infty}\frac{1}{8 \pi}  \int_{\Sigma_s}  (H_0 - \bar H) \, d \sigma_s  \\
  = & m_{ADM}(\tilde g) - m .
 \end{split}
\]
In the last equality, we use the theorem that the large sphere limit of the Brown-York mass on an asymptotically flat initial data is the total ADM energy \cite{Fan-Shi-Tam}. Combining the above, we have
\[
 \frac{1}{8 \pi } \int_{\Sigma_0} V (H_0 - H) \, d \sigma  \ge  m_{ADM}(\tilde g) - m
\]

On the other hand, we glue together $(M,g)$ with $(M_{ext}, \tilde g)$ and apply the Penrose inequality \cite{Miao}, we conclude that 
\[
 m_{ADM}(\tilde g)  \ge \sqrt{ \left(\frac{|\Sigma_h|}{16 \pi}\right) }.
\]
The proposition follows from combining the above two inequalities together.
\end{proof}
\section{The unit normal foliation}
In this section, we study the unit normal foliation in a spherically symmetric static manifold $\M$ and derive a sufficient condition on $\Sigma_0$ such that the foliation satisfies \eqref{condition_foliation}. We observe that it suffices to show that 
\[
\frac{\partial V}{\partial \nu}>0, \ \frac{\bar R}{2}- \bar Ric (\nu,\nu)  <  0   \ \mathrm{and} \  det (A_0) >  \frac{\bar R}{2}
\]
since $\bar R \ge T.$  The second condition follows from 
\begin{equation}\label{condition_1}
 \bar Ric (\nu,\nu) < 0   
\end{equation}
since $\bar R \ge 0$. 

On the other hand,
\begin{equation}\label{condition_2}
det(A_0)>  \frac{\bar R}{2}
 \end{equation}
follows from a suitable lower bound on the principle curvature of $\Sigma_s$. 

Finally, we observe that if $\bar Ric (\partial_r, \partial_r) < 0$, $V$ is spherical symmetric and $\frac{\partial V}{\partial r}>0$, then \eqref{condition_1} implies
\begin{equation}\label{V_increase}
\frac{\partial V}{\partial \nu}>0.
\end{equation}

A spherically symmetric manifold is conformally flat. Namely, there exist a coordinate such that 
\[
\bar g = F^4(\rho) (d\rho^2 + \rho^2 dS^2).
\] 
We identify a foliation $\Sigma_s$ of $\M$ with a foliation $\tilde \Sigma_s$ of $\R^3$ using this coordinate. Under the identification, a unit normal flow is identified with a flow with a given normal speed depending on the conformal factor. This allows us to find a condition on $\Sigma_0$ to guarantee  \eqref{condition_foliation} by studying the evolution of principal curvatures and the support function of surfaces in $\R^3$. We will first illustrate this idea using the Schwarzschild manifold since the conformal factor is explicit which simplifies the presentation. In the second subsection, we discuss the flow in the Reissner-Nordstrom manifold. In the last subsection, we prove a similar result for any asymptotically flat and spherically symmetric static manifold.
\subsection{The Schwarzschild case}
Let $m>0$. Recall that the Ricci curvature of
\[
\frac{dr^2}{1- \frac{2m}{r}} + r^2 dS^2
\]
in the orthonorgmal frame $\{ e_r, e_1 ,e_2 \}$ is 
\begin{equation} \label{Ricci_Schwarzschild}
\begin{pmatrix} 
- \frac{2m}{r^3} & 0 & 0 \\
0   &  \frac{m}{r^3} & 0\\
0 & 0 &  \frac{m}{r^3} 
\end{pmatrix}.
\end{equation}
We use the isothermal coordinate for the Schwarzschild manifold 
\[
(1+ \frac{m}{2 \rho})^4 (d \rho^2 + \rho^2 dS^2)
\]
which is obtained using a change of variable $\rho = \rho(r)$.  We set
\[
F = 1+ \frac{m}{2 \rho}.
\]

We can then identify a star-shaped surface $\Sigma$ in the Schwarzschild manifold defined by 
\[
\rho = G(u^a)
\]
with the surface $\tilde \Sigma$ in $\R^3$ defined by the same equation. As a result, a foliation  $\Sigma_s$ of the Schwarzschild manifold  corresponds to a foliation $\tilde \Sigma_s$ of $\R^3$. In particular,  $\Sigma_s$ is the unit normal foliation if and only if  $\tilde \Sigma_s$ satisfies 
\begin{equation} \label{flow_conformal}
\frac{\partial X}{\partial s}= \frac{\nu}{F^2} .
\end{equation}
where $\nu$ is the unit normal of $\tilde \Sigma_s$ in $\R^3$.

On $\tilde \Sigma_s$, we consider the support function 
\[
\langle X , \nu \rangle
\]
and the principal curvature $\tilde \kappa_a$. Let 
\[
\cos \theta = \frac{\langle X , \nu \rangle}{\rho}
\]
 \begin{proposition}
 If the surface $\tilde \Sigma_0$ satisfies
 \[
 \begin{split}
 \cos \theta > & \frac{1}{\sqrt{3}}\\
\min_{a}  \tilde \kappa_a  > &  \frac{\sqrt{3}  m }{\rho^2} \\
\rho   > &   3 m, 
 \end{split}
 \]
then the unit normal flow $\Sigma_s$ from $\Sigma_0$ in the Schwarzschild manifold satisfies \eqref{condition_1}, \eqref{condition_2} and \eqref{V_increase}.
 \end{proposition}
\begin{proof}
In the following, we restrict our discussion to the trajectory of a point on $\tilde \Sigma_0$ along the flow. By \eqref{Ricci_Schwarzschild},  \eqref{condition_1}  is the same as 
\[
\cos \theta  > \frac{1}{3}.
\]
Under the flow defined by \eqref{flow_conformal}, we have
\[
\begin{split}
\frac{d}{ds} \rho & = \frac{\cos \theta}{F^2} \\
\frac{d}{ds}  \cos \theta & \ge \frac{\sin^2 \theta}{F^2 \rho} -  |\nabla \frac{1}{F^2}| \\
 & = \frac{\sin \theta}{F^2 \rho}  (\sin \theta -  \frac{m}{F\rho})
\end{split}
\]
We claim that, if on $\Sigma_0$, we have 
\[
 \cos \theta > \sqrt{\frac{1}{3}} \ {\rm and } \ \rho > 3m,
\]
then for all $\Sigma_s$, we have
\[
 \cos \theta > \sqrt{\frac{1}{3}}.
\]
We prove by contradiction. Suppose it does not hold on all $s$. Then there is a first time $s= s_0$ that it is violated. At $s = s_0$, we have
\[
0 \ge \frac{d}{ds}  \cos \theta  \ge \frac{\sin \theta}{F^2 \rho}  (\sin \theta -  \frac{m}{F\rho})  > \frac{\sin \theta}{F^2 \rho} (\sqrt{\frac{2}{3}} - \frac{1}{3})  > 0.
\]
This is a contradiction. We conclude that $ \cos \theta > \sqrt{\frac{1}{3}}$ for all $s$. It follows that
\[
\frac{\partial V}{\partial \nu} > 0.
\]

On the other hand, \eqref{condition_2} holds if
\[
\tilde \kappa_a + 2 \frac{\partial_\nu F}{F} > 0.
\]
Since we need this for both principal curvatures, we simply write it as 
\[
\tilde \kappa > - 2 \frac{\partial_\nu F}{F}.
\]
A sufficient condition is that 
\[
\tilde \kappa  >  \frac{m}{\rho^2} .
\]
Let $A$ be the second fundamental form of $\tilde \Sigma_s$ in $\R^3$. Under the flow \eqref{flow_conformal}, we have
\[
\frac{d}{ds}  A_{ab} = - \nabla_a \nabla_b \frac{1}{F^2} +  \frac{1}{F^2}A_{ac} \sigma^{cd}A_{db}
\]
where
\[
 - \nabla_a \nabla_b \frac{1}{F^2} = - D_a D_b \frac{1}{F^2} + A_{ab}  \partial_\nu \frac{1}{F^2}.
\]
The last term is positive assuming the surface is convex and $\cos \theta > 0$. By a direct computation, we have
\[
  - D_a D_b \frac{1}{F^2} =   -\frac{m}{F^3 \rho^3 } \sigma_{ab}  + m  \frac{D_a \rho D_b \rho}{F^3 \rho^3} ( 3 - \frac{3m}{2 \rho F}) \ge -\frac{m}{F^3 \rho^3 } \sigma_{ab}  \ge -\frac{m}{F^2 \rho^3 } \sigma_{ab}.
\]
Thus we have
\[
\begin{split}
\frac{d}{ds}  ( \tilde  \kappa \rho^2)  \ge &  \frac{1}{\rho F^2} \left (    2 \rho^2 \cos \theta \tilde \kappa  - \rho^3 \tilde \kappa^2 - m \right )\\
 \ge & \frac{1}{\rho F^2} \left (    \frac{2}{\sqrt{3}} \rho^2 \tilde  \kappa  - \rho^3 \tilde \kappa^2 - m \right )
\end{split}
\]
where $ \cos \theta > \frac{1}{\sqrt{3}}$ is used in the last inequality. We claim that if 
\[
\begin{split}
\tilde \kappa \rho^2 >& \sqrt{3}  m \\
 \rho > & 3 m
 \end{split}
 \]
on $\tilde \Sigma_0$. Then 
\[
\tilde  \kappa \rho^2 > \sqrt{3}  m
\]
on all $\tilde \Sigma_s$. By the previous discussion on $\rho$, it suffices to show that the condition
\[
\tilde \kappa \rho^2 > \sqrt{3}  m 
\]
will be preserved along the flow. We prove by contradiction. Suppose it does not hold for all $s$. Then there is a first time $s= s_0$ that it is violated. At $s = s_0$, we have
\[
\tilde  \kappa \rho^2 = \sqrt{3}  m
\]
with 
\[
\frac{d}{ds } \tilde \kappa \rho^2 \le 0.
\]
However, we compute at $s=s_0$
\[
\begin{split}
\frac{d}{ds}  (\tilde  \kappa \rho^2)   > & \frac{1}{\rho F^2} \left (    \frac{2}{\sqrt{3}} \rho^2 \tilde  \kappa  - \rho^3 \tilde \kappa^2 - m \right )\\
 = & \frac{1}{\rho^2 F^2} \left (    \frac{2}{\sqrt{3}} \rho^2  \tilde \kappa - \frac{(\rho^2\tilde \kappa)^2}{\rho}- m \right )  \\
  = & \frac{1}{\rho^2 F^2} \left (    2m  - \frac{3 m^2}{\rho}- m \right ).  \\
\end{split} 
\]
 The last line is positive if $\rho > 3 m $. This is a contradiction. This finishes the proof.
\end{proof}
We formulate the above conditions in terms of conditions on $\Sigma_0$.
 \begin{corollary}
 If the surface $\Sigma_0$ satisfies 
 \[
 \begin{split}
- \bar Ric (\nu,\nu)  >  &0\\
\min_{a}  \kappa_a  > &  \frac{\sqrt{3}  m }{\rho^2} \\
\rho   > &   3 m,
 \end{split}
 \]
then the unit normal flow  $\Sigma_s$ from $\Sigma_0$ in the Schwarzschild manifold satisfies \eqref{condition_1}, \eqref{condition_2} and \eqref{V_increase}.
 \end{corollary}
\begin{proof}
We mentioned before that $\cos \theta >  \frac{1}{\sqrt{3}}$ is the same as $\bar Ric (\nu,\nu) <  0$. $\rho$ takes the same value on $\Sigma_0$ and  $\tilde \Sigma_0$. Finally, the principal curvature in the Schwarzschild manifold is smaller than the principal curvature in $\R^3$.
\end{proof}
 \subsection{The Reissner-Nordstrom case}
 In this subsection, let $(\M,\bar g)$ be the Reissner-Nordstrom manifold  with mass $m>0$ and charge $e$. In spherical coordinate, the metric is 
\[
-(1 - \frac{2m}{r} + \frac{e^2}{r^2}) dt^2 + \frac{d r^2}{1 - \frac{2m}{r} + \frac{e^2}{r^2}} + r^2 dS^2.
\]
and the Ricci curvature in the orthonorgmal frame $\{ e_r, e_1 ,e_2 \}$ is 
\begin{equation} \label{Ricci_Reissner-Nordstrom}
\begin{pmatrix} 
- \frac{2m}{r^3}  + \frac{e^2}{r^4}& 0 & 0 \\
0   &  \frac{m}{r^3} & 0\\
0 & 0 &  \frac{m}{r^3} 
\end{pmatrix}
\end{equation}

Similar to the last subsection, we express the metric in terms of the isothermal coordinate
 \[
 \bar g = F^4(\rho)(d \rho^2 + \rho^2 dS^2).
 \]
by a change of variable $\rho = \rho (r)$. While the change of variable is not explicit, we can still check that
\[
\partial_\rho F < 0.
\]

Using the isothermal coordinate, we identify a surface $\Sigma$ in the Reissner-Nordstrom manifold with a surface $\tilde \Sigma$ in $\R^3$ as in the last subsection.
\begin{proposition}\label{foliation_creterion}
There exists constant $C_1 > 0$ and $C_2>0$ (depending only on $m$ and $e$) such that  if the surface $\Sigma_0$ satisfies
\[
 - \bar Ric (\nu, \nu)  > 0 
\]
and the surface $\tilde \Sigma_0$ satisfies that 
 \[
 \begin{split}
\min_a \tilde \kappa_a  > &  \frac{C_1 }{\rho^2} \\
\rho   > &  C_2,
 \end{split}
 \]
 then the unit normal flow $\Sigma_s$ from $\Sigma_0$ in the Reissner-Nordstrom manifold satisfies \eqref{condition_1}, \eqref{condition_2} and \eqref{V_increase}. 
 \end{proposition}
\begin{proof}
On $\tilde \Sigma_s$, define
\[
\cos \theta = \frac{\langle X , \nu \rangle}{\rho}
\]
As before, $\cos \theta > 0$ implies that 
\[
\frac{\partial V}{\partial \nu} > 0. 
\]
From \eqref{Ricci_Reissner-Nordstrom}, $ - \bar Ric (\nu, \nu) > 0 $ if and only if 
\[
\cos \theta > \sqrt{ \frac{m   }{3m- \frac{e^2}{r(\rho)}}}.
\]
Let 
\[
G(\rho) = \sqrt{ \frac{m   }{3m- \frac{e^2}{r(\rho)}}}.
\]
It is a decreasing function of $\rho$.

Let $C_3$, $C_4$ and $C_5$ be positive constant such that 
\[
\begin{split}
|D \frac{1}{F^2}|  < &\frac{C_3}{F^2 \rho^2+1} \\
2 \frac{ | DF |}{F}  + F^2 \sqrt{R(\bar g)}  < &  \frac{C_4}{\rho^2+1} \\
D_a D_b \frac{1}{F^2} < & \frac{C_5}{\rho^3F^2+1}.
\end{split}
\]

In the following, we restrict our discussion to the trajectory of a point on $\tilde \Sigma_0$ along the flow. Let $\rho_0$ be the value of $\rho$ at the point we choose. Under the flow with normal speed $\frac{1}{F^2}$, we have
\[
\begin{split}
\frac{d}{ds} \rho & = \frac{\cos \theta}{F^2} \\
\frac{d}{ds}  \cos \theta & \ge \frac{\sin^2 \theta}{F^2 \rho} - |\nabla \frac{1}{F^2}| \\
 &  \ge \frac{\sin \theta}{F^2 \rho}  (\sin \theta -  \frac{C_3}{\rho})
\end{split}
\]
Thus, if on $\tilde \Sigma_0$, we have
\[
\begin{split}
\cos \theta > & G (\rho_0) \\
\rho  > & \frac{C_3}{\sqrt{1- G^2(\rho_0)}},
\end{split}
\]
then we claim that $\rho$ is increasing in $s$ and
\[
 \cos \theta >  G(\rho_0)
\]
holds on all $\Sigma_s$. Suppose it does not hold on all $s$. Then there is a first time $s= s_0$ that it is violated. At $s = s_0$, we have
\[
0 \ge \frac{d}{ds}  \cos \theta  \ge \frac{\sin \theta}{F^2 \rho}  (\sqrt{1- G^2(\rho_0)}-  \frac{C_3}{\rho})   > 0
\]
which is a contradiction. We conclude that $\cos \theta >  G (\rho_0) $. In particular, we have
\[
 \cos \theta >  \frac{1}{\sqrt{3}}.
\]

On the other hand, \eqref{condition_2} holds if
\[
\frac{\tilde \kappa_a}{F^2} + 2 \frac{\partial_\nu F}{F^3} > \sqrt{R(\bar g)}.
\]
A sufficient condition is that 
\[
\tilde \kappa_a >  \frac{C_4}{\rho^2} 
\]
for both principal curvature. In the following, we denote it simply by $\tilde \kappa$. We compute as before that
\[
\begin{split}
\frac{d}{ds}  (\tilde \kappa \rho^2)  \ge &  \frac{1}{\rho F^2} \left (    2 \rho^2 \cos \theta \tilde \kappa  - \rho^3\tilde \kappa^2 - C_5 \right )\\
> &  \frac{1}{\rho F^2} \left (   \frac{ 2}{\sqrt{3}} \rho^2 \tilde\kappa  - \rho^3 \kappa^2 - C_5 \right ).
\end{split}
\]
Set
\[
C = \max \{ C_4, C_5\}.
\]
Using the same proof by contradiction as the Schwarzschild case, we conclude that if 
\begin{equation}
\begin{split} 
\tilde \kappa \rho^2 >& \sqrt{3}  C \\
 \rho > & 3 C
 \end{split}
 \end{equation}
 on $\tilde \Sigma_0$, then the same holds on all $\tilde \Sigma_s$. As a result, the proposition is true if we set
 \[
 \begin{split}
 C_1  = &  \sqrt{3} \max \{ C_4, C_5\}  \\
 C_2  = &  \max\{  \frac{C_3}{\sqrt{1- (\max G)^2}}, \sqrt{3} C_4, \sqrt{3} C_5 \},
 \end{split}
 \]
where we take the maximum of $G$ among points outside the horizon. 
 \end{proof}
As a corollary, we have the following:
\begin{corollary}\label{corollary_foliation}
There exists constant $C_6 > 0$ and $C_7>0$ such that  if $\Sigma_0$ satisfies
\[
 \begin{split}
 - \bar Ric (\nu, \nu)  > & 0  \\
\min_a  \kappa_a  > &  \frac{C_6}{r^2}, \\
r  > &  C_7
 \end{split}
 \]
 then the unit normal flow $\Sigma_s$ from $\Sigma_0$ in the Reissner-Nordstrom manifold satisfies \eqref{condition_1}, \eqref{condition_2} and \eqref{V_increase}. 
 \end{corollary}
\begin{proof} 
As before, we have 
\[
 \kappa_a > \tilde \kappa_a.
\]
Moreover, $\rho (r)$ is an increasing function of $r$ such that 
\[
\rho = r + O(1).
\]
In particular, for $\rho > C_2$, there exists $C > c > 0$ such that 
\[
C r > \rho > cr.
\]
\end{proof}
Combining Proposition \ref{local_penrose_pre} and Corollary \ref{corollary_foliation}, we obtain the following theorem:
\begin{theorem}  
Let  $(M, g)$ be a compact, connected,  orientable, 3-manifold  with nonnegative scalar curvature, with boundary $\partial M$. 
Suppose the boundary is the union of  $ \Sigma $ and $ \Sigma_h$, where 
\begin{enumerate}
\item[(i)]   $\Sigma$ has positive mean curvature $H > 0$; and  
\item[(ii)]  $\Sigma_h$,  if nonempty, is a minimal surface and there are no other closed  minimal surfaces in $ (M, g)$.  
\end{enumerate}
There exists constant $C_6$ and $C_7$ depending only on the mass $m$ and charge $e$ of the reference Reissner-Nordstrom manifold
such that if $ \Sigma$ is isometric to a convex surface in the Reissner-Nordstrom manifold with principle curvature $\kappa_a$  and 
\[
\begin{split}
 -  \bar Ric (\nu,\nu) > & 0\\
\min \kappa_a > &  \frac{C_6}{r^2}\\
r > &  C_7.
\end{split}
\]
Then 
\[ 
 \frac{1}{8 \pi } \int_\Sigma V (H_0 - H) \, d \sigma  \ge \sqrt{ \left(\frac{|\Sigma_h|}{16 \pi}\right) } - m.
\]
\end{theorem}
\subsection{Spherically symmetric static manifold}
First, we define precisely the conditions we need for the static manifold. 
\begin{definition}
$\M$ is called a spherically symmetric reference manifold if it is a complete, asymptotically flat and spherically symmetric static manifold where the coordinate spheres are strictly mean convex, with possibly a minimal surface boundary. Furthermore, we assume that 
\[
\frac{\partial V}{\partial r} > 0,  \   \frac{\partial F}{\partial r} < 0  \ {\rm and}  \  \bar Ric (\partial_r,\partial_r) < 0,
\]
where $r$ is the area radius, $V$ is the static potential and $F$ is the conformal factor.
\end{definition}
\begin{remark}
By spherical symmetry, we also require that $V$ is a spherical symmetric function. Recall we also need that $V$ approaching to $1$ at infinity. This holds if $\M$ is the static slice of a spherical symmetric and asymptotically flat spacetime. 
\end{remark}
\begin{remark}
If there is a coordinate sphere with negative mean curvature, we can simply look at the region of $\M$ outside the outermost minimal surface.
\end{remark}

We identify surfaces $\Sigma_s$ in $\M$ with surfaces $\tilde \Sigma_s$ in $\R^3$ as before. From $ \bar Ric (\partial_r,\partial_r) < 0$, $\bar R >0$ and the spherical symmetry, we conclude that
\[
\bar Ric (\nu,\nu)  < 0
\]
in and only if
\[
\cos \theta  > G(\rho) > 0
\]
for some function $G(\rho)<1$. We define the following:
\begin{definition}
We say that a surface $\Sigma_0$ in $\M$ satisfies the tangent angle condition if 
\[
\cos \theta  > \max_{R \ge \rho_0}G(R).
\]
where $\rho_0$ denotes the restriction on $\rho$ to $\Sigma_0$.
\end{definition}
\begin{remark}
If $G$ is decreasing in $\rho$, then the tangent angle condition is the same as
\[
\bar Ric (\nu,\nu)  < 0.
\]
\end{remark}

We are now ready to state the quasi-local Penrose inequality
\begin{theorem}  
Suppose $\M$ is a spherically symmetric reference manifold. Let  $(M, g)$ be a compact, connected,  orientable, 3-manifold  with nonnegative scalar curvature, with boundary $\partial M$. 
Suppose the boundary is the union of  $ \Sigma $ and $ \Sigma_h$, where 
\begin{enumerate}
\item[(i)]   $\Sigma$ has positive mean curvature $H > 0$; and  
\item[(ii)]  $\Sigma_h$,  if nonempty, is a minimal surface and there are no other closed  minimal surfaces in $ (M, g)$.  
\end{enumerate}
There exists constant $C_8$ and $C_9$ depending only on $\M$ such that if $ \Sigma$ is isometric to a convex surface in $\M$ with principle curvature $\kappa_a$  such that
\[
\begin{split}
\min \kappa_a  > &  \frac{C_8}{r^2}\\
r > &  C_9
\end{split}
\]
and the tangent angle condition holds on $\Sigma_0$, then 
\[ 
 \frac{1}{8 \pi } \int_\Sigma V (H_0 - H) \, d \sigma  \ge \sqrt{ \left(\frac{|\Sigma_h|}{16 \pi}\right) } - m.
\]
where $m$ denote the ADM mass of $\M$.
\end{theorem}
\begin{proof}
The proof is almost identical to the Reissner-Nordstrom case in the last subsection. However, we do not know that $G(\rho)$ is decreasing and thus we assume 
\[
\cos \theta  > \max_{R \ge r}G(R)
\]
on $\Sigma_0$. We get constants corresponding to $C_3$, $C_4$ and $C_5$ which depend on the conformal factor $F$ of $\M$. The rest of the proof is the identical by tracking the inequalities along the flow.
\end{proof}

\end{document}